\documentclass[12pt,a4paper]{amsart}
\usepackage{graphicx,multirow,array,amsmath,amssymb,caption,kotex}

\newtheorem{theorem}{Theorem}
\newtheorem{proposition}[theorem]{Proposition}

\newtheorem*{example}{Example}

\begin{document}

\title{Linking numbers of Montesinos links}

\author[H. Kim]{Hyoungjun Kim}
\address{Institute of Data Science, Korea University, Seoul 02841, Korea}
\email{kimhjun@korea.ac.kr}

\author[S. No]{Sungjong No}
\address{Department of Mathematics, Kyonggi University, Suwon 16227, Korea}
\email{sungjongno@kgu.ac.kr}

\author[H. Yoo]{Hyungkee Yoo}
\address{Department of Mathematics Education, Sunchon National University, Sunchon 57922, Korea}
\email{hyungkee@scnu.ac.kr}
	

\begin{abstract}
The linking number of an oriented two-component link is an invariant indicating how intertwined the two components are.
Tuler proved that the linking number of a two-component rational $\frac{p}{q}$-link is 
$$\sum^{\frac{|p|}{2}}_{k=1} (-1)^{\big\lfloor (2k-1) \frac{q}{p} \big\rfloor }.$$
In this paper, we provide a simple proof the above result, and introduce the numerical algorithm to find linking numbers of rational links.
Using this result, we find linking numbers between any two components in a Montesinos link.
\end{abstract}

\maketitle

\section{Introduction}\label{sec:int}

All definitions and explanations throughout this paper are within the piecewise linear category or the differentiable category.
A {\it link\/} is a disjoint union of one dimensional spheres embedded in $\mathbb{R}^3$ or $S^3$.
In here, each one dimensional sphere is called a {\it component\/} of the link.
If a link has only one component, then we call it a {\it knot\/}.
In 1830's,
Gauss introduced an integral invariant for a two-component link.
Consider an orientated two-component link.
We can regard the above link as two embeddings
$K_1, K_2 : S^1 \hookrightarrow \mathbb{R}^3$.
Then there is a map $\Gamma : S^1 \times S^1 \rightarrow S^2$ defined by
$$\Gamma(s, t) = \frac{K_1(s)-K_2(t)}{|K_1(s)-K_2(t)|}.$$
This function is called the Gauss map.
Then the degree of the Gauss map is called the {\it linking number\/} of $K_1$ and $K_2$, and is denoted by $lk(K_1, K_2)$.
That is,
$${\rm lk}(K_1, K_2)
=\frac{1}{4\pi}\oint_{K_1}\oint_{K_2}
\frac{\mathbf{r}_1-\mathbf{r}_2}{|\mathbf{r}_1-\mathbf{r}_2|^3}\cdot(d\mathbf{r}_1 \times d\mathbf{r}_2)$$
where $\mathbf{r}_1$ and $\mathbf{r}_2$ are the positions of two points on $K_1$ and $K_2$ respectively~\cite{R}.
Since the degree does not change under any homotopic maps,
the linking number is the invariant for oriented two-component links.
If a link has more than two components,
then the linking number of the different component pairs are defined in the same way.

For any knots and links in $\mathbb{R}^3$,
there is a regular projection onto $\mathbb{R}^2$.
The regular projection image with crossing information of a link is called a {\it link diagram\/}.
For any link diagram, we can choose an orientation of each component.
In the oriented link diagram,
we assign the positive sign to a crossing between different link components if overstrand passes through understrand from left to right as drawn in the left side of Figure~\ref{fig:linking}.
Otherwise, we assign the negative sign.
It is known that the half of sum of all signs in the oriented link diagram is equal to the linking number of the link~\cite{A, R}.
For example, the linking number of a link in the right side of Figure~\ref{fig:linking} is equal to 1.

\begin{figure}[h!]
\centering
\includegraphics{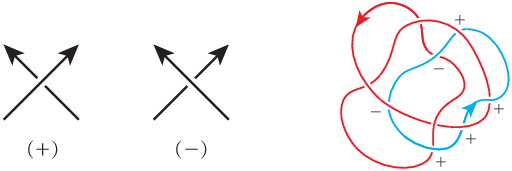}
\caption{Assigning the signs of crossings in a link.}
\label{fig:linking}
\end{figure}

In this paper, we deal with linking numbers of Montesinos links which consist of rational tangles and half-twists.
In Section~\ref{sec:pre}, we introduce a rational tangle and its pillowcase form.
We also give the structure of Montesinos links.
In Section~\ref{sec:link}, we provide the simple proof to find linking numbers of rational links.
Furthermore, we introduce the numerical algorithm for this result.
Finally, in Section~\ref{sec:monte}, we find linking numbers of Montesinos links.

\section{Preliminary}\label{sec:pre}

A {\it 2-tangle\/} is a proper embedding of the disjoint union of two arcs in a three-dimensional ball such that four endpoints of two arcs lie on the boundary sphere of the ball.
We label these four endpoints NW, NE, SW and SE as in the compass directions.
A {\it rational tangle} is a 2-tangle obtained by using horizontal twists and vertical twists from $T_0$ as drawn in Figure~\ref{fig:ele_tangle}.
In this process, we obtain a sequence of integers indicating how much horizontal twists and vertical twists have been used.
A tangle $T(a_1,a_2,\dots,a_n)$ is obtained from $T_0$ by using $a_1$ times of horizontal twists, then $a_2$ times of vertical twists, and repeat until $a_n$.
This is called the {\it Conway notation}.
The fraction of rational tangle $T(a_1,a_2,\dots,a_n)$ is
$$\cfrac{p}{q}=a_n + \cfrac{1}{a_{n-1} + \cfrac{1}{ \ddots +\cfrac{1}{a_2+ \cfrac{1}{a_1} }}}$$
where $p$ and $q$ are relatively prime.
A rational $\frac{p}{q}$-tangle is a tangle whose fraction is $\frac{p}{q}$.
Conway~\cite{C} showed that two rational tangles $T(a_1,a_2,...,a_n)$ and $T(a'_1,a'_2,...,a'_m)$ 
are equivalent if and only if
$$a_n + \cfrac{1}{a_{n-1} + \cfrac{1}{ \ddots +\cfrac{1}{a_2+ \cfrac{1}{a_1} }}}=a'_m + \cfrac{1}{a'_{m-1} + \cfrac{1}{ \ddots +\cfrac{1}{a'_2+ \cfrac{1}{a'_1} }}}.$$

\begin{figure}[h!]
\centering
\includegraphics[width=0.9\textwidth]{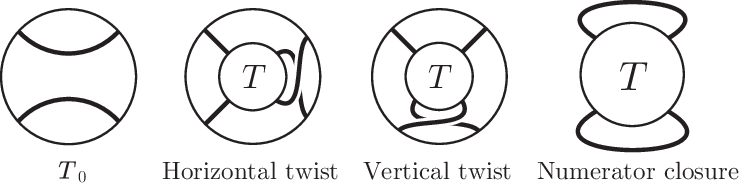}
\caption{Operations for rational tangles.}
\label{fig:ele_tangle}
\end{figure}

In a rational tangle $T$, the strand with NW as an endpoint can have the remaining three points SW, SE, and NE as an endpoint, depending on the relation between $p$ and $q$.
If such a strand has SW, SE and NE as an endpoint, then a rational tangle $T$ is called $V$-, $D$- and $H$-{\it tangle\/}, respectively.
Remark that a $\frac{p}{q}$-tangle is
a $V$-tangle if and only if $p$ is odd and $q$ is even;
a $D$-tangle if and only if $p$ is odd and $q$ is odd;
an $H$-tangle if and only if $p$ is even and $q$ is odd.
A {\it numerator closure} $N(T)$ is a way to obtain a rational link from a rational tangle $T$ by joining two pairs of its endpoints (NW,NE) and (SW,SE) together as drawn in Figure~\ref{fig:ele_tangle}.

A {\it pillowcase form} is one of the expressions of the rational tangle that allows two arcs of the tangle lie on the boundary sphere of the tangle as drawn in Figure~\ref{fig:pcf}.
A {\it $(t,s)$-form} is a pillowcase form which has $|t|$ and $|s|$ gaps in arcs connecting two top points and two left side points, respectively.
In detail, all parts of strands in the front side have the positive slopes when $ts$ is positive, and have the negative slopes when $ts$ is negative.
It is known that if $(t,s)$ is equal to $(p,q)$, then $(t,s)$-form is equivalent to the rational $\frac{p}{q}$-tangle~\cite{C,KN}.

\begin{figure}[h!]
\centering
\includegraphics[width=0.8\textwidth]{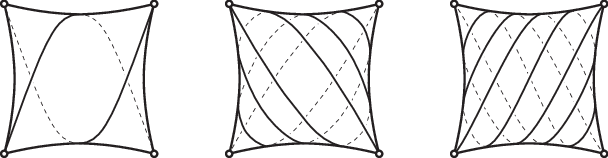}
\caption{Pillowcase forms of rational tangles with $\frac{2}{1}$, $-\frac{4}{3}$ and $\frac{5}{3}$.}
\label{fig:pcf}
\end{figure}

A {\it Montesinos link} $L$ is a link consisting of $n$ rational tangles and $e$ half twists connected in series as drawn in Figure~\ref{fig:monte}.
This link is denoted by 
$$L=M \left( \left. \frac{p_1}{q_1}, \frac{p_2}{q_2}, \cdots, \frac{p_n}{q_n} \right| e \right) $$
where each tangle $T_i$ is a rational $\frac{p_i}{q_i}$-tangle.
In here, if $e$ is negative, then the part of $e$ half twist consists of negative $|e|$ half twists.

\begin{figure}[h!]
\centering
\includegraphics{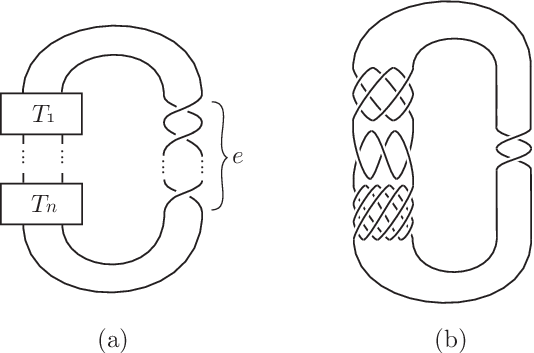}
\caption{The structure of a Montesinos link and the example of $ M \left( \left. \frac{3}{2}, -\frac{3}{1}, \frac{5}{3} \right| 2 \right)$.}
\label{fig:monte}
\end{figure}

If the Montesinos link $L$ has exactly one or two rational tangles, then $L$ is a 2-bridge link.
Thus it is sufficient to consider the case that $n \ge 3$ when we discuss about Montesinos links.

\section{Linking numbers of rational links}\label{sec:link}

In this section, we deal with the linking number of a rational $\frac{p}{q}$-link.
Let $R_{p/q}$ be an oriented rational link obtained from a rational $\frac{p}{q}$-tangle consisting two strands such that the orientation of one strand is from NE to NW, and the orientation of the other is from SW to SE by using a numerator closure.
In 1981, Tuler~\cite{T} showed the following proposition.
Using the pillowcase form, we present another simple proof of this proposition.

\begin{proposition}\label{prop:link}
The linking number of an oriented rational link $R_{p/q}$ is
$$lk(R_{p/q})=\sum^{\frac{|p|}{2}}_{k=1} (-1)^{\big\lfloor (2k-1) \frac{q}{p} \big\rfloor}$$
where $p$ and $q$ are relatively prime integers and $p$ is even.
\end{proposition}

\begin{proof}[Alternative proof of Proposition~\ref{prop:link}]
Let $p$ and $q$ be integers satisfying the assumption of the proposition,
and let $T$ be a rational $\frac{p}{q}$-tangle.
Since $p$ is even,
the numerator closure $N(T)$ of $T$ has two components.
Let $A$ and $B$ be oriented arcs from NE to NW, and from SW to SE, respectively.

Consider the pillowcase form $P$ of $T$ as drawn in Figure~\ref{fig:pillow} (a).
As the figure, we assume that $\alpha$ is an oriented arc from NE to NW in the boundary of $P$.
Since there is no obstruction in the interior of $P$,
we can shrink $A$ to $\alpha$ with fixing end points as drawn in Figure~\ref{fig:pillow} (b).
To compute the linking number of the rational link $N(T)$,
it is sufficient to observe the intersection of $B$ and $\alpha$.

\begin{figure}[h!]
\centering
\includegraphics{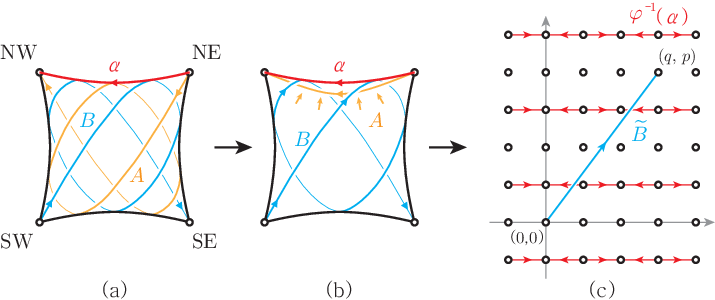}
\caption{The construction of the branched covering of $\partial P$.}
\label{fig:pillow}
\end{figure}

Now consider the branched covering of $\partial P$.
Let $S$ be the standard torus.
Then there is a 2-fold branched covering map
$$\varphi_1 : S \rightarrow \partial P$$
with branch set $\{$NE, NW, SE, SW$\}$.
This map is the quotient map by involution.
Since the universal cover of $S$ is $\mathbb{R}^2$,
there is a covering map
$$\varphi_2 : \mathbb{R}^2 \rightarrow S.$$
The preimage of four points $\{$NE, NW, SE, SW$\}$ under $\varphi_1 \circ \varphi_2$ is $\mathbb{Z}^2$.
Thus $\varphi = \varphi_1 \circ \varphi_2$
becomes a branched covering map with branch set $\{$NE, NW, SE, SW$\}$.
Without loss of generality, we may assume that $(0,0)$, $(0,1)$, $(1,0)$ and $(1,1)$ of $\mathbb{R}^2$ are corresponding to $SW$, $SE$, $NW$ and $NE$ under $\varphi$, respectively.
Then the line segment between $(0,0)$ and $(q,p)$ is the lifting $\Tilde{B}$ of the arc $B$ in $\partial P$.
The preimage of $\alpha$ is $$\displaystyle{\varphi^{-1}(\alpha)=\bigcup_{n,m \in \mathbb{Z}} \big( (n,n+1) \times \{ 2m+1 \} \big)}.$$
Since $\alpha$ is an oriented arc from NE to NW,
the orientation of $\varphi^{-1}(\alpha)$ appears alternately as drawn in Figure~\ref{fig:pillow} (c).
We assign the positive sign to an intersection point between $\Tilde{B}$ and $\varphi^{-1}(\alpha)$ if $\Tilde{B}$ crosses from left to right based on the direction of $\varphi^{-1}(\alpha)$ at the point, and the negative sign for otherwise.
Thus the sign of $k$-th crossing between $\Tilde{B}$ and $\varphi^{-1}(\alpha)$ is 
$$(-1)^{\big\lfloor (2k-1) \frac{q}{p} \big\rfloor}.$$

We remark that each intersection point between $\Tilde{B}$ and $\varphi^{-1}(\alpha)$ including the information of the assigned sign corresponds to a full-twist between $A$ and $B$ in Figure~\ref{fig:pillow} (b).
By the definition of the linking number, it is sufficient to find that the sum of signs of all intersection points between $\Tilde{B}$ and $\varphi^{-1}(\alpha)$ to obtain the linking number of $N(T)$.
Thus we obtain the result.
\end{proof}

Proposition~\ref{prop:link} shows a rigorous mathematical proof of how to obtain the linking number.
However, the calculation may be complicated depending on the value of $|p|$.
We provide the numerical algorithm to calculate linking numbers of rational links by using Proposition~\ref{prop:alg} which is simpler than the previous result.

\begin{proposition}\label{prop:alg}
For an oriented rational link, the following are satisfied.
\begin{enumerate}
    \item $lk(R_{0/1}) = 0$.
    \item $lk(R_{p/q}) = -lk(R_{(-p)/q})$.
    \item $lk(R_{p/q}) = lk(R_{(p+2q)/q}) -1$.
    \item $lk(R_{p/q}) = -lk(R_{p/(p+q)})$.
\end{enumerate}
\end{proposition}

\begin{proof}
(1) A rational link $R_{0/1}$ is trivial.
So $lk(R_{0/1})=0$.

(2) Since $R_{(-p)/q}$ has an embedding which is a mirror image of $R_{p/q}$, $lk(R_{p/q}) = -lk(R_{(-p)/q})$.

(3) Let $R_{p/q}$ be a rational link obtained from a rational tangle $T$ as drawn in Figure~\ref{fig:twist}.
We further assume that the fraction of the tangle $T$ is
$$\cfrac{p}{q}=a_n + \cfrac{1}{ \ddots +\cfrac{1}{a_2+ \cfrac{1}{a_1}}}.$$

First suppose that $R_{p'/q'}$ is a rational link obtained from a rational tangle $T'$ as drawn in Figure~\ref{fig:twist} (a).
This means that the tangle $T'$ is obtained from the tangle $T$ by adding two horizontal twists.
Thus the fraction $\frac{p'}{q'}$ of the tangle $T'$ is
$$\cfrac{p'}{q'}=(a_n +2) + \cfrac{1}{ \ddots +\cfrac{1}{a_2+ \cfrac{1}{a_1} }} = 2 + \cfrac{p}{q} = \cfrac{p+2q}{q}.$$
Furthermore, since the linking number increases by 1 at the two added horizontal twists, $lk(R_{p'/q'})=lk(R_{p/q})+1$.

(4) Now suppose that $R_{p''/q''}$ is a rational link obtained from a rational tangle $T''$ as drawn in Figure~\ref{fig:twist} (b).
This means that the tangle $T''$ is obtained from the tangle $T$ by adding one vertical twist.
Thus the fraction $\frac{p''}{q''}$ of the tangle $T''$ is
$$\cfrac{p''}{q''}=0 + \cfrac{1}{1+\cfrac{1}{a_n + \cfrac{1}{\ddots+ \cfrac{1}{a_1} }}} = 0 + \cfrac{1}{1+\cfrac{q}{p}} = \cfrac{p}{p+q}.$$
Note that the only different crossings between $R_{p/q}$ and $R_{p''/q''}$ is a self crossing.
Moreover, the assigned signs of all crossings in $R_{p/q}$ are reversed in $R_{p''/q''}$.
Hence $lk(R_{p/q}) = -lk(R_{p/(p+q)})$.
\end{proof}

\begin{figure}[h!]
\centering
\includegraphics[width=0.57\textwidth]{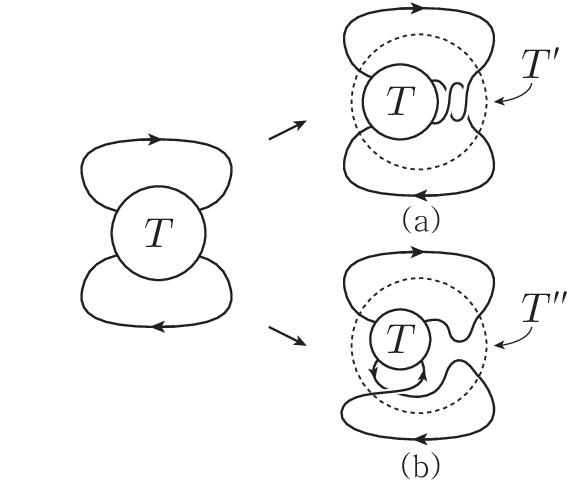}
\caption{Two types of adding twists}
\label{fig:twist}
\end{figure}

Proposition~\ref{prop:alg} shows the difference of the linking number between a rational link obtained from the given tangle and a transformed tangle by some rules.
By using this proposition, we can easily obtain the linking number of a rational link.
For a rational link $R_{2n/1}$, 
$$lk(R_{2n/1})=lk(R_{2n-2/1})+1= \cdots = lk(R_{0/1})+n=n.$$
The following example shows the two ways to obtain linking numbers of rational links that are Tuler's result and the process of Proposition~\ref{prop:alg}.

\begin{example}
The linking number of $R_{26/9}$.
\end{example}

First we use Tuler's result to obtain the linking number of an oriented rational link $R_{26/9}$.
$$\begin{array}{rl} 
lk(R_{26/9}) = & \sum^{13}_{k=1} (-1)^{\big\lfloor (2k-1) \frac{9}{26} \big\rfloor} \\[8pt]
= & (-1)^{\big\lfloor \frac{9}{26} \big\rfloor}+(-1)^{\big\lfloor \frac{27}{26} \big\rfloor}+(-1)^{\big\lfloor \frac{45}{26} \big\rfloor}+(-1)^{\big\lfloor \frac{63}{26} \big\rfloor}\\ 
& + (-1)^{\big\lfloor \frac{81}{26} \big\rfloor}+ (-1)^{\big\lfloor \frac{99}{26} \big\rfloor}+(-1)^{\big\lfloor \frac{117}{26} \big\rfloor}+(-1)^{\big\lfloor \frac{135}{26} \big\rfloor}\\
&+ (-1)^{\big\lfloor \frac{153}{26} \big\rfloor}+(-1)^{\big\lfloor \frac{171}{26} \big\rfloor}+(-1)^{\big\lfloor \frac{189}{26} \big\rfloor}+(-1)^{\big\lfloor \frac{207}{26} \big\rfloor}\\
&+ (-1)^{\big\lfloor \frac{225}{26} \big\rfloor}\\[8pt]
= & (-1)^{0}+(-1)^{1}+(-1)^{1}+(-1)^{2}+ (-1)^{3}\\
& +(-1)^{3}+(-1)^{4}+(-1)^{5}+(-1)^{5}+(-1)^{6}\\
& +(-1)^{7}+(-1)^{7}+ (-1)^{8}\\[6pt]
= & -3
\end{array}$$

Now we use the process of Proposition~\ref{prop:alg} to obtain the linking number of an oriented rational link $R_{26/9}$.
\begin{align*} 
lk(R_{26/9}) & = lk(R_{8/9})+1\\ 
 & = -lk(R_{8/1})+1 \\
 & = -(lk(R_{0/1})+4)+1=-3
\end{align*}

This example shows that the process of Proposition~\ref{prop:alg} makes it easier to obtain linking numbers of rational links.

\section{Linking numbers of Montesinos links}\label{sec:monte}\

In this section,
we deal with linking numbers of Montesinos links.
Let $L$ be a Montesinos link
$M \left( \left. \frac{p_1}{q_1}, \frac{p_2}{q_2}, \cdots, \frac{p_n}{q_n} \right| e \right)$.
Recall that there are three types of rational tangles, $V$-, $D$- and $H$-tangle, which depend on locations of endpoints of two strands.
We divide into two cases according to the existence of $H$-tangles in $L$.

\subsection{The case there is an $H$-tangle in $L$}\

We assume that $L$ contains an $H$-tangle.
If $L$ contains exactly one $H$-tangle, then $L$ is a 1-component link.
So we only need to consider the case that $L$ contains at least two $H$-tangles.
Note that all crossings between different components are only appeared in $H$-tangles.
We may assume that $T_i$ is an $H$-tangle such that two components $C$ and $C'$ in $L$ are linked at $T_i$.
If there is another $H$-tangle $T_j$ such that $C$ and $C'$ are linked at $T_j$, then $L$ has exactly two components $C$ and $C'$.
This implies that if $L$ has at least three $H$-tangles, then $C$ and $C'$ are only linked at $T_i$.
We first give the following theorem when $L$ has at least three $H$-tangles.

\begin{theorem}\label{thm:mpe}
Let $L=M \left( \left. \frac{p_1}{q_1}, \frac{p_2}{q_2}, \cdots, \frac{p_n}{q_n} \right| e \right)$ be a Montesinos link. Further assume that $C$ and $C'$ are distinct components of $L$. 
If there are at least three $H$-tangles in $L$, then the linking number of $C$ and $C'$ is
$$lk(C,C') = \begin{cases} \pm lk(R_{p_k/q_k}) & \text{if } C \text{ and } C' \text{ are linked at the } \frac{p_k}{q_k} \text{-tangle } \\
0 & \text{otherwise,} \end{cases}$$
where $R_{p_k/q_k}$ is the rational $\frac{p_k}{q_k}$-link.
\end{theorem}

\begin{proof}
Let $C$ and $C'$ be distinct components of a Montesinos link $L$.
First suppose that $C$ and $C'$ are not linked. 
Hence the linking number between $C$ and $C'$ is equal to zero.

Now suppose that $C$ and $C'$ are linked.
Note that rational tangles of $L$ are connected in cyclic order.
Since $L$ has at least three $H$-tangles, we may assume that $C$ and $C'$ are only linked at an $H$-tangle, say a rational $\frac{p_k}{q_k}$-tangle.
Note that the linking number between $C$ and $C'$ is determined by crossings between $C$ and $C'$.
So, it sufficient to check a rational $\frac{p_k}{q_k}$-tangle to obtain $lk(C,C')$.
This means that $C \cup C'$ has the same  linking number as $R_{p_k/q_k}$ in absolute value.
Therefore, $lk(C,C')= \pm lk(R_{p_k/q_k})$.
\end{proof}

It remains to consider the case that $L$ contains exactly two $H$-tangles.
Since $C$ and $C'$ are linked at exactly two $H$-tangles, we should consider all crossings in these two tangles.
We give the following theorem when $L$ has exactly two $H$-tangles.

\begin{theorem}\label{thm:tqe}
Let $L=M \left( \left. \frac{p_1}{q_1}, \frac{p_2}{q_2}, \dots, \frac{p_n}{q_n} \right| e \right)$ be a Montesinos link such that $L$ has exactly two $H$-tangles $T_i$ and $T_j$.
Then the linking number of $L$ is
$$lk(L) = \pm \left( lk(R_{p_i/q_i})+(-1)^{\sigma}lk(R_{p_j/q_j}) \right),$$
where $\displaystyle{\sigma = \sum_{k=1}^n q_k +e}$.
\end{theorem}

\begin{proof}
Let $L$ be a Montesinos link with $n$ rational $\frac{p_i}{q_i}$-tangles $T_i$ and $e$ half twists.
For convenience of proof,
we regard the part of $e$ half twists to be a tangle $T_{n+1}$.
Then $T_{n+1}$ is a $D$-tangle if $e$ is odd, and is a $V$-tangle if $e$ is even.
Let $C_1$ and $C_2$ be distinct components of Montesinos link $L$.
We further assume that $C_1$ and $C_2$ are linked at $T_i$ and $T_j$ for $i<j$.
Without loss of generality, we may assume that $C_1$ is a component containing $T_{n+1}$, and $C_2$ is the other.

Choose an orientation of $L$ such that $T_i$ has an orientation as drawn in Figure~\ref{fig:H-tangles}.
Then the orientation of the other $H$-tangle $T_j$ is determined by the number of $D$-tangles of $C_1$ and $C_2$.
In detail, the orientation of $C_1$ in $T_j$ is from SW to SE if $C_1$ has even number of $D$-tangles, and from SE to SW for otherwise.
Furthermore, the orientation of $C_2$ in $T_j$ is from NE to NW if $C_2$ has even number of $D$-tangles, and from NW to NE for otherwise.

\begin{figure}[h!]
\centering
\includegraphics{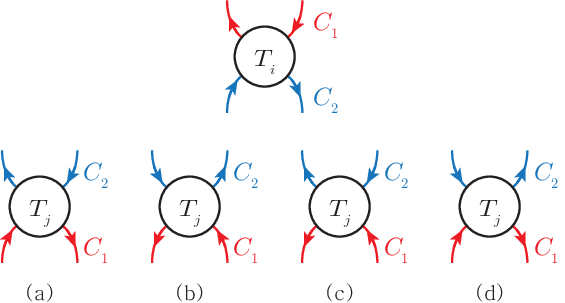}
\caption{Four cases of the orientation of $T_j$.}
\label{fig:H-tangles}
\end{figure}

Remark that the linking number between $C_1$ and $C_2$ is determined by the signs of crossings between them.
The crossing between $C_1$ and $C_2$ are only appeared in $T_i$ and $T_j$.
The sum of signs of all crossings in $T_i$ is equal to $lk(R_{p_i/q_i})$.
For the sum of signs of all crossings in $T_j$, we consider the orientations of two strands consisting of $T_j$.

Now we have four candidates for the orientation of $C_1$ and $C_2$ in $T_j$ as drawn in Figure~\ref{fig:H-tangles}.
In the figure (a) and (b), the sum of signs of all crossings in $T_j$ is equal to $lk(R_{p_j/q_j})$, and there are even number of $D$-tangles among $T_1,\dots,T_{n+1}$.
In the figure (c) and (d), the sum of signs of all crossings in $T_j$ is equal to $-lk(R_{p_j/q_j})$, and there are odd number of $D$-tangles.
Note that $T_{n+1}$ is a $D$-tangle when $e$ is odd, and a $V$-tangle when $e$ is even.
So the number of $D$-tangles is congruent to  $\displaystyle{\sum_{k=1}^n q_k +e}$ modulo 2.
Thus the linking number of $L$ with the chosen orientation is
$$ lk(R_{p_i/q_i})+(-1)^{\sigma}lk(R_{p_j/q_j}),$$
where $\displaystyle{\sigma = \sum_{k=1}^n q_k +e}$.
The absolute value of the linking number is not affected by the orientation of each component.
Therefore, the linking number of $L$ is
$$lk(L) = \pm \left( lk(R_{p_i/q_i})+(-1)^{\sigma}lk(R_{p_j/q_j}) \right).$$
\end{proof}

\subsection{The cases there is no $H$-tangle in $L$}\

We assume that $L$ does not contain any $H$-tangle.
If $\sigma$ is odd then $L$ is 1-component link where $\sigma=\displaystyle{\sum_{k=1}^n {q_k} + e}$.
So we only need to consider that $\sigma$ is even.

\begin{theorem}\label{thm:aqo}
Let $L=M \left( \left. \frac{p_1}{q_1}, \frac{p_2}{q_2}, \cdots, \frac{p_n}{q_n} \right| e \right)$ be a 2-component Montesinos link.
If there is no $H$-tangle in $L$, then the linking number of $L$ is
$$lk(L) = \pm \left( \sum_{k=1}^n lk(R_{(q_k+e_k p_k)/p_k})+\cfrac{1}{2} \left(e-\sum_{k=1}^n e_k \right ) \right),$$
where $e_i=\cfrac{1-(-1)^{q_i}}{2}$.
\end{theorem}

\begin{proof}
Let $L$ be a 2-component Montesinos link which consists of $n$ rational $\frac{p_i}{q_i}$-tangles $T_i$ with $e$ half twists.
We further assume that there is no $H$-tangle in $L$.
Choose an orientation of $L$ as drawn in Figure~\ref{fig:case1}~(a).

\begin{figure}[h!]
\centering
\includegraphics{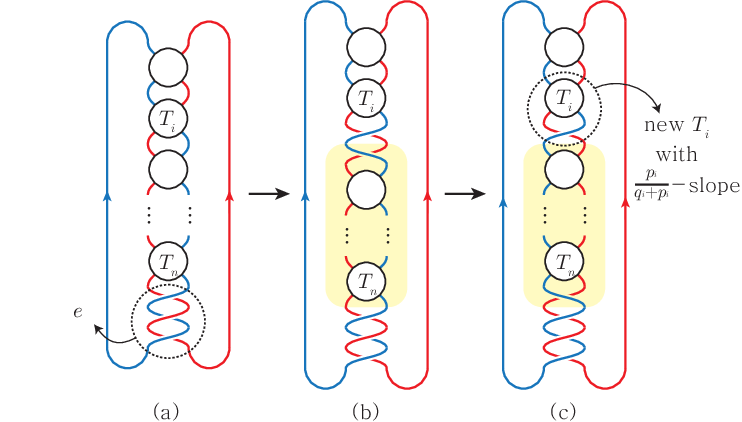}
\caption{A 2-component Montesinos link with no $H$-tangle.}
\label{fig:case1}
\end{figure}

First we transform $L$ into the form that contains no $D$-tangle.
If there is a $D$-tangle $T_i$, then we transform it to $V$-tangle as the following process.
We add a positive half-twist and a negative half-twist below $T_i$ by using a type II Reidemeister move as drawn in Figure~\ref{fig:case1}~(b).
Use a flype the shaded region of the figure, which contains the area from the added negative twist to the tangle $T_n$.
Since the rational tangle can be represented by the pillowcase form, every tangle in the shaded region is not changed after the flype.
Then the slope of $T_i$ is changed from $\frac{p_i}{q_i}$ to $\frac{p_i}{q_i+p_i}$, and one negative half twist is added below the $T_n$ as drawn in Figure~\ref{fig:case1}~(c).
Repeat this process until there is no $D$-tangle.
Then the slope of every rational tangle is changed from $\frac{p_i}{q_i}$ to $\frac{p_i}{q_i+e_i p_i}$ and $\displaystyle{\sum_{k=1}^n e_k}$ negative half twists are added below the $T_n$, where $e_i=\frac{1-(-1)^{q_i}}{2}$.
Remark that $e_i$ is equal to 0 when $q_i$ is even, and 1 when $q_i$ is odd.
Furthermore, by combining the added negative twists and the part of $e$ half twists, there are $e-\displaystyle{\sum_{k=1}^n e_k}$ half twists below the $T_n$.
Therefore, $L$ can be regarded as a new expression,
$$L=\displaystyle{M \left( \left. \frac{p_1}{q_1+e_1p_1}, \dots, \frac{p_n}{q_n+e_np_n} \right| e-\sum_{k=1}^n e_k \right)}.$$

Henceforth, we use the new expression of $L$ which consists of only $V$-tangles.
Then one component of $L$ passes through every tangle from NW to SW and the other from NE to SE.
Consider a rational $\frac{p}{q}$-tangle as a pillowcase $(p,q)$-form as drawn in Figure~\ref{fig:tangles}~(a).
Rotate the tangle through an angle $\frac{\pi}{2}$ in the clockwise direction.
Then $(p,q)$-form is changed to $(-q,p)$-form as drawn in Figure~\ref{fig:tangles}~(b).
Note that a rational link with $(-q,p)$-form is $R_{(-q)/p}$.
This means that the sum of signs of all crossings between different components of $(-q,p)$-form is equal to $-2lk(R_{q/p})$ when the pillowcase form has the same orientation as Figure~\ref{fig:tangles}~(c).

\begin{figure}[h!]
\centering
\includegraphics{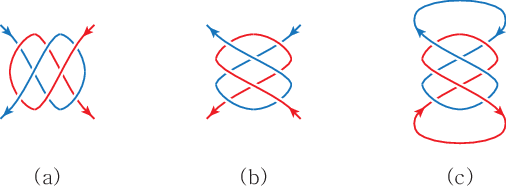}
\caption{(a) a pillowcase form of $\frac{p}{q}$-tangle,
(b) rotated $\frac{p}{q}$-tangle,
(c) the oriented rational link $R_{(-q)/p}$.}
\label{fig:tangles}
\end{figure}

Considering the difference in orientation between a tangle in Figure~\ref{fig:tangles}~(a) and a link in Figure~\ref{fig:tangles}~(c), the sum of signs of all crossings between different components in the tangle is equal to $2lk(R_{q/p})$.
Therefore the sum of signs of all crossings between different components in $L$ is equal to
$$\sum_{k=1}^n 2lk(R_{(q_k+e_k p_k)/p_k})+ \left(e-\sum_{k=1}^n e_k \right ).$$
The absolute value of the linking number is not affected by the orientation of each component.
Therefore, the linking number of $L$ is
$$lk(L) = \pm \left( \sum_{k=1}^n lk(R_{(q_k+e_k p_k)/p_k})+\cfrac{1}{2} \left(e-\sum_{k=1}^n e_k \right ) \right).$$
\end{proof}

\section*{acknowledgement}

The first author(Hyoungjun Kim) was supported by the National Research Foundation of Korea (NRF) grant funded by the Korea government Ministry of Science and ICT(NRF-2021R1C1C1012299 and NRF-2022M3J6A1063595).
The corresponding author(Sungjong No) was supported by the National Research Foundation of Korea(NRF) grant funded by the Korea government Ministry of Science and ICT(NRF-2020R1G1A1A01101724).
The third author (Hyungkee Yoo) was supported by Basic Science Research Program of the National Research Foundation of Korea (NRF) grant funded by the Korea government Ministry of Education (RS-2023-00244488).

\end{document}